\documentclass{article}

\usepackage{upgreek}
\usepackage{amssymb,amsmath,amsfonts,amsthm,mathrsfs,commath} 

\usepackage{lmodern}
\usepackage{bbm}
\usepackage{calc,extarrows}
\usepackage{enumitem}
\usepackage{graphicx}
\usepackage{graphics}
\usepackage{mathrsfs}
\usepackage{fullpage}
\usepackage{url}
\usepackage[dvipsnames]{xcolor}
\usepackage{tikz-cd}
\usepackage{float}
\usepackage{setspace}
\usepackage{todonotes}
\usepackage{orcidlink}
\usepackage{url}
\usepackage{titling}
\usepackage{comment}

\usepackage{pgfplots}
\usetikzlibrary{arrows.meta}
\usetikzlibrary{plotmarks}
\pgfplotsset{compat=1.15}
\usepackage{mathrsfs}
\usetikzlibrary{arrows}

\definecolor{bluepigment}{rgb}{0.2, 0.2, 0.6}
\definecolor{amethyst}{rgb}{0.6, 0.4, 0.8}
\definecolor{asparagus}{rgb}{0.53, 0.66, 0.42}
\definecolor{orange-red}{rgb}{1.0, 0.27, 0.0}
\definecolor{coffee}{rgb}{0.88, 0.82 ,0.77}
\definecolor{champagne}{rgb}{0.97, 0.91, 0.81}
\definecolor{deepchampagne}{rgb}{0.98, 0.84, 0.65}
\definecolor{bittersweet}{rgb}{0.8, 0.58, 0.46}
\definecolor{antiquebrass}{rgb}{0.8, 0.58, 0.46}
\definecolor{burlywood}{rgb}{0.87, 0.72, 0.53}
\definecolor{darkchampagne}{rgb}{0.76, 0.7, 0.5}
\definecolor{bisque}{rgb}{1.0, 0.89, 0.77}
\definecolor{biscuit}{RGB}{244, 164, 96}
\definecolor{darkchestnut}{rgb}{0.6, 0.41, 0.38}

\usepackage{hyperref} 
\hypersetup{
	colorlinks=true,
	linkcolor=amethyst,
	citecolor=orange-red,
	filecolor=black,      
	urlcolor=black,
}

\usepackage[utf8]{inputenc} 

\newcommand{\Z}{\mathbb{Z}}
\newcommand{\R}{\mathbb{R}}

\newcommand{\calA}{\mathcal{A}}
\newcommand{\calB}{\mathcal{B}}
\newcommand{\calC}{\mathcal{C}}

\newcommand{\calE}{\mathcal{E}}

\newcommand{\calN}{\mathcal{N}}

\newcommand{\calX}{\mathcal{X}}

\newcommand{\floor}[1]{{\left\lfloor #1 \right\rfloor}}

\DeclareMathOperator{\E}{\mathbb{E}}
\renewcommand{\P}{\operatorname{\mathbb{P}}}

\newcommand{\e}{\mathrm{e}}

\renewcommand{\bar}[1]{\overline{#1}}

\let\oldepsilon\epsilon
\let\epsilon\varepsilon
\let\varepsilon\oldepsilon

\renewcommand{\phi}{\varphi}

\theoremstyle{plain}

\newtheorem{theorem}{Theorem}[section]

\newtheorem{lem}[theorem]{Lemma}
\newtheorem{prop}[theorem]{Proposition}

\theoremstyle{definition}

\numberwithin{equation}{section}

\theoremstyle{remark}
\newtheorem{remark}[theorem]{Remark}
\title{Coexistence for competing branching random walks with identical asymptotic shape on $\mathbb{Z}^d$}

\author{
Partha Pratim Ghosh
    \orcidlink{0000-0002-4801-4538}
    \thanksgap{0.05ex}
    \thanks{Fakultät für Mathematik, Ruhr-Universität Bochum,
Universitätsstraße 150,
44801 Bochum, Germany}
    \\p.pratim.10.93@gmail.com
\and
Benedikt Jahnel 
    \orcidlink{0000-0002-4212-0065}
    \thanksgap{0.05ex}
    \thanks{Institut für Mathematische Stochastik, Technische Universit\"at Braunschweig, Universit\"atsplatz 2, 38106 Braunschweig, Germany}
    \thanksgap{0.2ex} 
    \thanks{Weierstrass Institute for Applied Analysis and Stochastics, Anton-Wilhelm-Amo-Straße 39, 10117 Berlin, Germany} 
    \\ benedikt.jahnel@tu-braunschweig.de
}

\date{\today}

\begin{document}
\maketitle

\begin{abstract}
We consider two independent branching random walks that start next to each other on the $d$-dimensional hypercubic lattice and that carry two different colors. Vertices of the lattice are colored according to the color of the walker cloud that first visits the vertex, leading to the question of possible coexistence in the sense that both colors appear on infinitely many vertices.  Under mild conditions, we prove the coexistence for two independently distributed branching random walks obeying the same first- and second-order behavior for their extremal particles. To complement this result, we also exhibit examples for the almost-sure absence of coexistence, for $d=1$, in cases where the asymptotic shapes of the walker clouds are calibrated to coincide, thereby answering a question by Deijfen and Vilkas (ECP  28(15):1--11, 2023). As a main tool we employ second-order and large-deviation approximations for the position of the extremal particles in one-dimensional branching random walks. 

\smallskip
\noindent\footnotesize{{\textbf{AMS-MSC 2020}: Primary: 60D05; Secondary: 60K35}

}
\end{abstract}

\section{Introduction}\label{Sec:intro}
{\em Branching random walks} (BRWs) form a canonical model for populations that both reproduce and disperse in space.  In the simplest one‐type setting on the lattice \(\Z^d\), each particle independently produces a random number of offsprings according to a given offspring distribution and places them at random displacements from its parent’s location.  Over the past decades a rich collection of limit theorems has been obtained for the growth and spatial spread of a single BRW. In particular, {\em laws of large numbers} and {\em shape theorems}~\cite{Bigg76,Hammersley1974,Kingman1975,Biggins1990}, {\em central limit theorems} for the front position~\cite{Biggins1992,Aidekon2013,Mallein2016},
and {\em large‐deviation principles} for generation sizes and maximal displacements~\cite{GantertDyszewskiHofelsauer2023,GantertHofelsauer2018,LouidorPerkins2015,Zhang2022,ShiWanlin2019}.

In many natural applications such as ecology, epidemiology, and genetics, different {\em types} of particles compete for resources or territory.  This motivates the studies of various types of {\em competing branching random walks}, in which two (or more) colors of particles branch and disperse as in a BRW. In this article, we focus on a version of competition in which any lattice site is colored according to the color of the particle that first visits that site, but otherwise BRWs for different colors behave independently. A first rigorous analysis of such competition on \(\Z^d\) was carried out in~\cite{DeVi23} for two different colors, say blue and red, and in situations in which the first-order behavior of the blue, respectively red, BRWs are different. More precisely, using the fact that BRWs have an asymptotic linear speed in every direction (in fact they even obey a shape theorem), they show that the existence of at least one advantageous direction in which, e.g., blue has a strictly larger asymptotic speed is enough to ensure that blue colors infinitely many site. {\em Coexistence} in the sense that both BRWs color infinitely many sites simultaneously can be guaranteed if both colors have at least one advantageous direction. If one color dominates in the sense that all direction have a strictly larger asymptotic speed, there is no coexistence in the sense that the slower color can only color finitely many sites with probability one, see~\cite[Proposition~1.2]{DeVi23}. 

In this manuscript we complement the findings above by considering cases in which the two colors have the same asymptotic speed in all directions. Our contribution is twofold:
\begin{enumerate}
    \item We show that, under rather general conditions, coexistence happens with probability one if the two BRWs are independently distributed and their extremal particles obey the same first and second-order behavior. This proves a corresponding conjecture in~\cite[Page~5]{DeVi23}.
    \item Already in one spatial dimension, second-order fluctuations in the extrema of the two processes are enough to guarantee non-coexistence.
\end{enumerate}
Our method of proof for coexistence is based on large-deviation bounds on atypically large excursions from the typical position of the extremal particles.  
More precisely, we show that in one spatial dimension,  for every $0<c_1<c_2$, almost~surely, for all sufficiently large $z$, there will be points of both colors in the interval $[\exp(c_1z),\exp(c_2z)]$. As a byproduct of our results we also establish that, for a one-dimensional BRW, if the underlying progeny point process contains finitely many points almost~surely, every individual born in the BRW has a descendant that reaches the right-most position at some future time, a result first proved in~\cite{LaSe87} for branching Brownian motions.
For the non-coexistence we exhibit a class of models that allow us to match the first-order behavior, i.e., the asymptotic speeds, while still creating a gap in the logarithmic second-order behavior. The main challenge that we then overcome is to ensure that potential holes in the annulus of sites between the extremal particles of both colors, are still colored first by the dominating BRW.

The manuscript is organized as follows. In the following Section~\ref{Sec:set} we first present the framework of BRWs based on a point-process description of the offspring distribution and our general assumptions. The main theorems then feature our coexistence and non-coexistence results. Finally, Section~\ref{sec:proof} contains all proofs.

\section{Setting and main results}\label{Sec:set}

Consider two time-discrete {\em branching random walks} (BRWs) on $\Z^d$, one colored red and the other blue, starting with a single red particle at
the origin $\boldsymbol{o}\in \Z^d$ and a single blue particle at $\boldsymbol{b}\in \Z^d$. The
evolutions of the two BRWs are governed by their offspring distributions, which
we assume to be (not necessarily simple) point processes
$\Phi := \sum_{j \ge 1} \boldsymbol{\delta}_{\chi_j}$
    and
    $\Psi := \sum_{j \ge 1} \boldsymbol{\delta}_{\varsigma_j}$
    on  $\Z^d$,
where $\boldsymbol{\delta}_x$ denotes the Dirac measure at $x$.
Let $\boldsymbol{e}$ be a unit vector in $\R^d$. We define the point processes obtained by projecting the offspring positions along
$\boldsymbol{e}$ as 
$\Phi_{\boldsymbol{e}}:= \sum_{j\geq 1}\boldsymbol{\delta}_{\langle \chi_j,\boldsymbol{e} \rangle}$ and $\Psi_{\boldsymbol{e}}:= \sum_{j\geq 1}\boldsymbol{\delta}_{\langle \varsigma_j,\boldsymbol{e} \rangle}$.

Next, in order to present further notation and conditions for the projected point processes $\Phi_{\boldsymbol{e}},\Psi_{\boldsymbol{e}}$, consider a point process $\Xi := \sum_{j \ge 1} \boldsymbol{\delta}_{\xi_j}$ on $\R$. The distribution of $\Xi$ is completely determined by its log-Laplace transform
\[
\kappa_\Xi(\theta):=  \log \E\Big[ \sum_{j\geq 1} \e^{\theta \xi_j} \Big],\qquad \theta\in \R.
\]
In order to describe the drift and variance of the process, we also introduce
\[
\kappa'_{\Xi}(\theta):=  \E\Big[ \sum_{j\geq 1} \xi_j \e^{\theta \xi_j-\kappa_{\Xi}(\theta)}\Big]\quad\text{and}\quad \kappa''_{\Xi}(\theta):= \E\Big[ \sum_{j\geq 1} \left(\xi_j-\kappa'_{\Xi}(\theta) \right)^2 \e^{\theta \xi_j-\kappa_{\Xi}(\theta)}\Big],\qquad \theta\in \R.
\]
Note that, under the assumption that $\kappa_{\Xi}$ is finite in an open neighborhood of $\theta$, by Lebesgue’s dominated convergence theorem, one readily checks that $\kappa_\Xi'(\theta)$ and $\kappa_\Xi''(\theta)$ exist, are finite and are indeed the first and second derivatives of $\kappa_\Xi$ at $\theta$. 

Throughout this work we assume that there exists a unit vector $\boldsymbol{e}\in\R^d$
such that each $\Xi \in \{\Phi_{\boldsymbol{e}}, \Psi_{\boldsymbol{e}}\}$ satisfies
the following technical conditions.
\begin{itemize}[leftmargin=1.5cm]
		\item[{\bf (A1)}] 
		$\Xi$ is \emph{non-trivial}, and the \emph{extinction probability} of the underlying \emph{branching process} is $0$. In other words,  
$\P( \Xi(\{a\}) = \Xi(\R) ) < 1$ for any $a \in \R$,  
$\P( \Xi(\R) = 0 ) = 0$, and  
$\P( \Xi(\R) = 1 ) < 1$.

        \item[{\bf (A2)}]
There exist $\eta > 0$ and $\theta_\Xi > 0$ such that 
$\kappa_\Xi(\theta) < \infty$ for all $\theta \in (-\eta, \theta_{\Xi} + \eta)$, 
$\kappa_\Xi'(\theta_{\Xi}) > 0$, and 
$\theta_{\Xi} \kappa_\Xi'(\theta_{\Xi}) = \kappa_\Xi(\theta_{\Xi})$.

        \item[{\bf (A3)}]
		The quantities $\E[W((\log W)_+)^2 ]$ and $\E[\bar{W}(\log\bar{W})_+]$ are finite, where  $W=\sum_{j\geq1}\e^{\theta_{\Xi} \xi_j-\kappa_\Xi(\theta_{\Xi})}$, \linebreak $\bar{W}=\sum_{j\geq1} (\kappa_\Xi'(\theta_{\Xi})-\xi_j)_+\e^{\theta_{\Xi} \xi_j-\kappa_\Xi(\theta_{\Xi})}$, and $x_+=\max(x,0)$.
\end{itemize}

\begin{remark}
   Note that Condition~{\bf (A1)} implies that $\kappa_\Xi$ is strictly convex. Consequently, if $\theta_{\Xi} \in (0,\infty)$ satisfying Condition~{\bf (A2)} exists, it is the unique point in $(0, \infty)$ such that a tangent line from the origin to the graph of $\kappa_\Xi(\theta)$ touches the graph at $\theta = \theta_{\Xi}$. Moreover, if $\E[\Xi(\R)^{1+\varepsilon}]$ is finite for some $\varepsilon > 0$, then Condition~{\bf (A3)} is also satisfied.
\end{remark}

\subsection{Coexistence}
We consider two independent time-discrete branching random walks
$\mathcal{X}^{(r)}$ and $\mathcal{X}^{(b)}$ on $\Z^d$ with starting positions
$\mathcal{X}^{(r)}_0 = \boldsymbol{\delta}_{\boldsymbol{o}}$ and
$\mathcal{X}^{(b)}_0 = \boldsymbol{\delta}_{\boldsymbol{b}}$ for
$\boldsymbol{b} \in \Z^d$, which are driven by the offspring distributions
$\Phi$ and $\Psi$, respectively, satisfying Conditions {\bf (A1)--(A3)} above for some $\boldsymbol{e}$.
Let $R \subseteq \Z^d$ and $B \subseteq \Z^d$ denote the sets of vertices that are
first visited by $\calX^{(r)}$ and $\calX^{(b)}$, respectively, where an arbitrary
tie-breaking rule is applied if a vertex is discovered simultaneously by both
BRWs. Our results do not depend on the choice of the tie-breaking rule.
 We may now state our first main result.
\begin{theorem}(Coexistence)\label{thm:coexistence}
Whenever
$\theta_{\Phi_{\boldsymbol{e}}}=\theta_{\Psi_{\boldsymbol{e}}}$ and
$\kappa_{\Phi_{\boldsymbol{e}}}(\theta_{\Phi_{\boldsymbol{e}}})
 = \kappa_{\Psi_{\boldsymbol{e}}}(\theta_{\Psi_{\boldsymbol{e}}})$,
we have, for all $\boldsymbol{b}\in\Z^d$, that
\[
    \P(R \text{ is infinite}) = \P(B \text{ is infinite}) = 1.
\]
In particular, for every $0 < c_1 < c_2$, almost~surely, for all sufficiently large $z$, the set $\{\boldsymbol{a} \in \mathbb{Z}^d \colon \exp(c_1 z)\le \langle \boldsymbol{a}, \boldsymbol{e} \rangle \le  \exp(c_2 z)\}$ contains vertices of both colors.
\end{theorem}
\begin{remark}
As a special case, coexistence occurs when $\Phi$ and $\Psi$ are identically distributed.
\end{remark}

Let $X^{(r)}$ and $X^{(b)}$ denote the projections of $\mathcal{X}^{(r)}$ and $\mathcal{X}^{(b)}$, respectively, along the vector $\boldsymbol{e}$. 
Then $X^{(r)}$ and $X^{(b)}$ are one-dimensional BRWs driven by the offspring
distributions $\Phi_{\boldsymbol{e}}$ and $\Psi_{\boldsymbol{e}}$, with initial
configurations $X^{(r)}_0 = \boldsymbol{\delta}_0$ and
$X^{(b)}_0 = \boldsymbol{\delta}_{\varrho}$, where
$\varrho := \langle \boldsymbol{b}, \boldsymbol{e} \rangle$.
When
\[
\theta_{\Phi_{\boldsymbol{e}}} = \theta_{\Psi_{\boldsymbol{e}}}, \qquad
\kappa_{\Phi_{\boldsymbol{e}}}(\theta_{\Phi_{\boldsymbol{e}}})
  = \kappa_{\Psi_{\boldsymbol{e}}}(\theta_{\Psi_{\boldsymbol{e}}}), \qquad
\kappa'_{\Phi_{\boldsymbol{e}}}(\theta_{\Phi_{\boldsymbol{e}}})
  = \kappa'_{\Psi_{\boldsymbol{e}}}(\theta_{\Psi_{\boldsymbol{e}}}),
\]
we write $\theta_o$, $\kappa(\theta_o)$, and $\kappa'(\theta_o)$ for these
common values, for notational convenience.

The proof of Theorem~\ref{thm:coexistence} is based on a large-deviation analysis for the unlikely displacement of the right-most particles in $X^{(b)}$, respectively $X^{(r)}$, at time $n\ge 0$, which we denote by 
$M_n^{(b)}$, respectively $M_n^{(r)}$. More precisely, we consider a {\em centering} of the right-most particles, at time $n\ge 0$, defined as
\[
m_n:=n\kappa'(\theta_o)-\frac{3}{2\theta_o}\log n.
\]
Further, we consider stopping times for the large-deviation event of overshooting the centering, i.e., 
\[
T^{(r)}(z):=\inf\{n\geq 0\colon  M_n^{(r)}-m_n > z \}
\qquad
\text{and}
\qquad
T^{(b)}(z):=\inf\{n\geq 0\colon M_n^{(b)}-m_n > z \}.
\]
The following statement establishes the asymptotic behavior of $T^{(r)}(z)$ in the limit of large $z$, which maybe of independent interest. 
\begin{prop}(First overshoot)\label{prop:Tr}
We have that almost~surely,
\begin{align*}
    \lim_{z\uparrow\infty} \frac{1}{z}\log T^{(r)}(z)=\theta_o.
\end{align*}
\end{prop}
\begin{remark}
As will become clear from the proof, Proposition~\ref{prop:Tr} applies to any
point process $\Xi$ on $\R$ satisfying  {\bf (A1)--(A3)}; in
particular, it does not rely on $\Xi$ being obtained as the projection of a
point process on $\Z^d$. This result is the BRW analogue of
the corresponding theorem for branching Brownian motion,
see~\cite[Theorem~1.1]{Chen13}. 
\end{remark}

As a byproduct of our analysis, we obtain the following result, which is the BRW
analogue of the corresponding statement for branching Brownian motion first
proved in~\cite{LaSe87}.
\begin{theorem}\label{Thm:Democracy}
Let $\Xi$ be any point process on $\R$ satisfying  {\bf (A1)--(A3)}.
Then, for the one-dimensional BRW driven by $\Xi$, almost~surely every
individual born in the process has a descendant that reaches the right-most position at some future time.
\end{theorem}

\subsection{Non-coexistence}
Our second main result features a class of models for competing BRWs that are independent and share the same first-order behavior, i.e., 
\begin{align}\label{Eq:FirstOrder}
\lim_{n\uparrow\infty}M_n^{(r)}/n=\lim_{n\uparrow\infty}M_n^{(b)}/n,\qquad\text{almost~surely},
\end{align}
but do not share the same centering. We show that differences in the second-order behavior are enough to exclude coexistence. In what follows, we focus our attention on pairs of independent BRWs $X^{(r)}$ and $X^{(b)}$ on $\Z$ with starting positions $X^{(r)}_0=0$ and $X^{(b)}_0=1$ that are based on
\begin{align*}
    \Xi_{r}=\sum_{j=1}^{N_{r}}\boldsymbol{\delta}_{\xi^{(r)}_j}\qquad \text{and}\qquad \Xi_{b}=\sum_{j=1}^{N_{b}}\boldsymbol{\delta}_{\xi^{(b)}_j},
\end{align*}
where $N_{r},N_{b}\in \{1,2,\dots\}$ are independent random variables with $1<\E[N_{r}]<\infty$ and $1<\E[N_{b}]<\infty$, so that both the BRWs survive almost~surely. We further assume that $(\xi^{(r)}_j)_{j\ge 1}$ and  $(\xi^{(b)}_j)_{j\ge 1}$ are families of i.i.d.\ symmetrically distributed random variables with values in $\{-M,\dots, M\}$ that are mutually independent and independent of $N_{r},N_{b}$ and with $\P(\xi^{(r)}_1=1)\wedge\P(\xi^{(b)}_1=1)>0$. Then, within the class of BRWs just defined, we can observe the following. 
\begin{prop}\label{Prop:counter}
For any non-degenerate distributions of $\xi^{(r)}_1$ and $\xi^{(b)}_1$, there exist distributions of $N_r$ and $N_b$ for which~\eqref{Eq:FirstOrder} holds and both $N_r$ and $N_b$ are almost~surely bounded.
\end{prop}
As a warm up, we can observe non-coexistence in the case when the blue particles only make unit steps, which is advantageous, since the blue particles create paths without holes. 
\begin{prop}\label{Prop:counter2}
Let $\P(\xi^{(b)}_1=1)=1/2$. Then, there exist distributions of $\xi^{(r)}_1$, $N_r$ and $N_b$ such that~\eqref{Eq:FirstOrder} holds, but
\begin{align*}
\P(B \text{ is infinite})=\P(R\text{ is finite})=1.
\end{align*}
\end{prop}
From the proof we can see that we need  $\E[(\xi^{(r)}_1)^2]> 3$. 
Our second main result now features non-trivial conditions under which there is no coexistence even if both paths have holes. 
\begin{theorem}\label{Thm:counter}
There exist distributions of $\xi^{(r)}_1$, $\xi^{(b)}_1$, $N_r$ and $N_b$ with
$\E[(\xi^{(r)}_1)^2] > 3\E[(\xi^{(b)}_1)^2]$ and 
$\P(\xi^{(r)}_1 \notin \{-1,0,1\})\wedge\P( \xi^{(b)}_1 \notin \{-1,0,1\}) > 0$
such that~\eqref{Eq:FirstOrder} holds, but
\begin{align*}
\P(B \text{ is infinite})=\P(R\text{ is finite})=1.
\end{align*}
\end{theorem}


\section{Proofs}\label{sec:proof}
\subsection{Proofs for coexistence}
Let us start with some preliminary large-deviation statements that we, together with all subsequent lemmas, prove at the end of the section. 
\begin{lem}\label{Lem:upper-bound-of-upper-deviation}
For all $\delta>0$, there exists $c,C>0$ such that for all sufficiently large $z$,
\begin{align}
   \P\big(\exists n \leq \e^{\theta_o z}\colon M_n^{(r)}-m_n > (1+\delta)z \text{ or } M_n^{(b)}-m_n > (1+\delta)z \big) 
   &<  C(1+(1+\delta)\theta_o z)\e^{-\theta_o\delta z}\quad\text{and}\label{dis_3}\\
   \P\big(\forall n \leq \e^{\theta_o z}\colon M_n^{(r)}-m_n < (1-\delta)z \text{ or } M_n^{(b)}-m_n < (1-\delta)z\big) 
   &<  C\e^{-c\delta z}.\label{dis_3b}
\end{align}
\end{lem}

With the help of Lemma~\ref{Lem:upper-bound-of-upper-deviation}, we can directly prove Proposition~\ref{prop:Tr}.
\begin{proof}[Proof of Proposition~\ref{prop:Tr}]
Note that, for all $\delta>0$,
\begin{align*}
    \P(|z^{-1}\log T^{(r)}(z)-\theta_o|>\delta)&=\P(T^{(r)}(z)<\e^{(\theta_o-\delta)z})+\P(T^{(r)}(z)>\e^{(\theta_o+\delta)z})\\
    &=\P\big(\exists n \leq \e^{(\theta_o -\delta)z}\colon M_n^{(r)}-m_n >z \big)+\P\big(\forall n \leq \e^{(\theta_o +\delta)z}\colon M_n^{(r)}-m_n >z \big),
\end{align*}
which is exponentially small by Lemma~\ref{Lem:upper-bound-of-upper-deviation} and hence summable in $z$. An application of the Borel--Cantelli lemma now gives the result. 
\end{proof}

Next, we quantify the relative positions of the right-most particles up to the stopping times.

\begin{prop}
\label{Prop:Maximum-advantage-of-blue-before-Trz}
For all $\delta\in(0,1)$, there exist $C_1,C_2>0$ such that for all sufficiently large $z$,
\begin{align}\label{Eq:dis_3}
    \P\big(\sup\big\{ M_n^{(b)} \colon n\leq T^{(r)}(z)\big\} > m_{T^{(r)}(z)}+\delta z\big)<C_1 \e^{-C_2\delta z}.
\end{align}
\end{prop} 
Before we present the proof of this statement, let us provide the proof of our first main result. 
\begin{proof}[Proof of Theorem~\ref{thm:coexistence}]
Observe that by the definition of $T^{(r)}(z)$, 
\[
M_{T^{(r)}(z)}^{(r)} >m_{T^{(r)}(z)} +z.
\]
By Proposition~\ref{Prop:Maximum-advantage-of-blue-before-Trz} and using the Borel--Cantelli lemma, we obtain that for $\delta \in (0,1)$, almost~surely, for all sufficiently large $z$,
\begin{align}\label{Equ:Blue-did-not-reached-MrTrz}
\sup\big\{ M_n^{(b)} \colon n\leq T^{(r)}(z)\big\} \le m_{T^{(r)}(z)}+\delta z < M_{T^{(r)}(z)}^{(r)}.   
\end{align}
Now, let $H^{(r)}(z):=\big\{ \boldsymbol{a}\in\Z^d \colon \langle \boldsymbol{a},\boldsymbol{e} \rangle 
= M_{T^{(r)}(z)}^{(r)} \big\}$.
We recall that for all $n\ge 0$, $X_n^{(r)}$ and $X_n^{(b)}$ are the projections of $\mathcal{X}_n^{(r)}$ and $\mathcal{X}_n^{(b)}$, respectively, along the vector $\boldsymbol{e}$. Therefore, we get that almost~surely, for all sufficiently large $z$,
\begin{align}
    \calX_{T^{(r)}(z)}^{(r)}\big(H^{(r)}(z) \big)\geq 1,
    \qquad\text{ but }\qquad
    \calX_i^{(b)}\big( H^{(r)}(z) \big)=0 \text{ for all $0\leq i\leq M_{T^{(r)}(z)}^{(r)}$},
\end{align}
which implies that almost~surely, for all sufficiently large $z$,
\begin{align}
    R\cap  H^{(r)}(z) \neq \emptyset.
    \label{Equ:RcapH}
\end{align}
Fix $\delta>0$. By Proposition~\ref{prop:Tr}, almost~surely, for all sufficiently large $z$, 
\begin{align}
  \e^{\theta_o (1-\delta)z}\leq T^{(r)}(z) \leq \e^{\theta_o (1+\delta)z} .
  \label{Equ:Bound-on-Trz}
\end{align}
In combination with Lemma~\ref{Lem:upper-bound-of-upper-deviation} and the
Borel--Cantelli lemma, this implies that almost~surely, for all sufficiently large $z$, 
\[
M_{T^{(r)}(z)}^{(r)} \leq m_{T^{(r)}(z)} +(1+2\delta)z.
\]
On the other hand, it follows from the definition that 
\[
M_{T^{(r)}(z)}^{(r)} \geq m_{T^{(r)}(z)} +z.
\]
Combining these with~\eqref{Equ:Bound-on-Trz}, we obtain that almost~surely, for all sufficiently large $z$, 
\begin{align}
  \kappa'(\theta_o)\e^{\theta_o (1-2\delta)z}  \leq M_{T^{(r)}(z)}^{(r)}\leq \kappa'(\theta_o) \e^{\theta_o (1+2\delta)z}.
  \label{Equ:Bound-on-MrTrz}
\end{align}
Let $q > (1+2\delta)/(1-2\delta)$ be an integer. Since the intervals
$[ \kappa'(\theta_o)\,\e^{\theta_o (1-2\delta) q^i},
            \kappa'(\theta_o)\,\e^{\theta_o (1+2\delta) q^i} ]$
are pairwise disjoint for all $i \ge 0$, it follows that, almost~surely, the
values $M_{T^{(r)}(q^i)}^{(r)}$ are distinct for all sufficiently large $i$.
Consequently, the sets $H^{(r)}(q^i)$ are almost~surely pairwise disjoint for
all sufficiently large $i$, which, together with~\eqref{Equ:RcapH}, implies that $R$ is infinite almost~surely.
Using the exact same argument we also see that $B$ is infinite almost~surely. Thus we proved the coexistence.

Now, for the second statement note that \eqref{Equ:RcapH} implies that, almost~surely, 
for all sufficiently large $z$,  $ R\cap  H^{(r)}(z) \neq \emptyset$. By an analogous argument, defining
$H^{(b)}(z):=\big\{ \boldsymbol{a}\in\Z^d \colon \langle \boldsymbol{a},\boldsymbol{e} \rangle 
= M_{T^{(b)}(z)}^{(b)} \big\}$, we get that $ B\cap  H^{(b)}(z) \neq \emptyset$
 almost~surely for all sufficiently large $z$. Fix any $0<c_1<c_2$. To prove the desired result, it is then enough to show that, almost~surely, 
for all sufficiently large $z$,  
 \begin{align}
 \label{Equ:location_estimate}
    H^{(r)}(z)\cup H^{(b)}(z)\subseteq\{\boldsymbol{a} \in \mathbb{Z}^d \colon \exp(c_1 z)\le \langle \boldsymbol{a}, \boldsymbol{e} \rangle \le  \exp(c_2 z)\}. 
 \end{align}
We choose $c,\delta>0$ such that
\[
    c_1 < \theta_o(1-2\delta)c 
       < \theta_o(1+2\delta)c 
       < c_2.
\]
From~\eqref{Equ:Bound-on-MrTrz} we see that, almost~surely, for all
sufficiently large $z$,
\[
\e^{c_1 z} < \kappa'(\theta_o)\e^{\theta_o (1-2\delta)\floor{cz}}  \leq M_{T^{(r)}(\floor{cz})}^{(r)}\leq \kappa'(\theta_o) \e^{\theta_o (1+2\delta)\floor{cz}} <\e^{c_2 z},
\]
and, similarly,
\[
\e^{c_1 z} <  M_{T^{(b)}(\floor{cz})}^{(b)} <\e^{c_2 z}.
\]
This implies~\eqref{Equ:location_estimate}  and completes the proof.
\end{proof}
In what follows, we write for an individual $u$ in the BRW, $|u|$ for its generation and $S_u$ for its position.
\begin{proof}[Proof of Theorem~\ref{Thm:Democracy}] 

Let $u$ be an individual in the BRW driven by the offspring point process $\Xi$,
and let $|u| = q$ be its generation number. Let $\mathsf N_{n}$ be the total number of individuals at generation $n$. By Assumption~{\bf (A2)}, 
$\E[\mathsf N_q] = \e^{q\kappa(0)} < \infty$.
For any $v$ in generation $q$, the descendants of $v$ form a BRW, denoted
$X^{(v)}$, defined by
\[
    X^{(v)}_k := \sum_{|w| = q+k,\, v \prec w} 
    \boldsymbol{\delta}_{(S_w - S_v)},
\]
where $v\prec w$ stands for $v$ being an ancestor of $w$.
The family $\{X^{(v)}\colon |v| = q\}$ consists of i.i.d.\ BRWs started at~$o$ and
driven by the same offspring point process $\Xi$.
For $z > 0$, set
\[
    T^{(u)}(z) := \inf\{ n \ge 0\colon  M^{(u)}_n - m_n > z \}.
\]
By Proposition~\ref{Prop:Maximum-advantage-of-blue-before-Trz}, for any
$\delta \in (0, 1/4)$ there exist constants $C_1, C_2 > 0$ such that, for each
$v \neq u$ with $|v| = q$,
\[
 \P\big(\sup\big\{ M_n^{(v)} \colon n\leq T^{(u)}(z)\big\} > m_{T^{(u)}(z)}+\delta z\big)<C_1 \e^{-C_2\delta z}.
\]
Define the event
\[
\calE_{1,z}:=\Big\{ \sup_{|v|=q,\,v\neq u}\big\{ M_n^{(v)} \colon n\leq T^{(u)}(z)\big\} > m_{T^{(u)}(z)}+\delta z \Big\}.
\]
We have that
\begin{align*}
   \P(\calE_{1,z}) 
   =\E\big[\P(\calE_{1,z}| \mathsf N_q )\big]< C_1\E[\mathsf N_q] \e^{-C_2\delta z} = C_1\e^{q\kappa(0)} \e^{-C_2\delta z}.
\end{align*}
Let $\mathcal{E}_{2,z}$ be the event that there exists an individual in
generation $q$ whose position lies outside the interval $[-\delta z, \delta z]$.
For any $\alpha \in (0, \eta)$, with $\eta$ as in Assumption~{\bf (A2)}, Markov’s
inequality yields
\begin{align*}
    \P(\calE_{2,z})
    \leq \P\Big(\sum_{|v|=q}\e^{-\alpha S_v} + \sum_{|v|=q}\e^{\alpha S_v}\geq \e^{\alpha\delta z}\Big) 
    \leq  \e^{-\alpha\delta z} \big(\e^{\kappa(-\alpha)}+ \e^{\kappa(\alpha)} \big).
\end{align*}
Then, by the Borel--Cantelli lemma, almost~surely the event
$(\mathcal{E}_{1,z} \cup \mathcal{E}_{2,z})^{\complement}$ occurs for all
sufficiently large~$z$. On this event we have
\begin{align*}
    \sup\big\{S_w\colon |w|=q+ T^{(u)}(z), u\nprec w\big\}
    &=\sup\{S_v + M_{T^{(u)}(z)}^{(v)}\colon |v|=q, v\neq u\}\\
    & \leq S_u +2\delta z + \sup\big\{ M_{T^{(u)}(z)}^{(v)}\colon |v|=q, v\neq u\big\}\\
    & \leq S_u +2\delta z + m_{T^{(u)}(z)} + \delta z\\
    &< S_u + 3\delta z+ M_{T^{(u)}(z)}^{(u)} -z\\
    & = \sup\big\{S_w\colon |w|=q+ T^{(u)}(z), u\prec w\big\} - (1-3\delta)z\\
    &<\sup\big\{S_w\colon |w|=q+ T^{(u)}(z), u\prec w\big\}.
\end{align*}
This shows that, almost~surely, for all sufficiently large $z$, the particles at
the right-most position at time $q + T^{(u)}(z)$ are descendants of~$u$. This
completes the proof.
\end{proof}
In the remainder of the section, we prove the supporting results. 
\begin{proof}[Proof of Lemma~\ref{Lem:upper-bound-of-upper-deviation} Part~\eqref{dis_3}]

Under our assumptions, \cite[Proposition~2.1]{Mada17} guarantees that there exists a constant $c_1>0$ such that, for any $n\geq 1$, $\beta>1$, and $x\geq 1$, 
\[
\P\Big( n^{3\beta/2}\sum_{|u^{(r)}|=n}\e^{\beta(\theta_o S(u^{(r)})-\kappa(\theta_o))}>\e^{\beta x}\Big)
  < c_1(1+x)\e^{-x}.
\]
Now, observing that 
\[
\e^{\beta(\theta_o M_n^{(r)}-\kappa(\theta_o))}
  \leq \sum_{|u^{(r)}|=n}\e^{\beta(\theta_o S(u^{(r)})-\kappa(\theta_o))}
\]
and taking
$C_1 = \max\{c_1,\e/2 \} $,
we obtain that, for any $n\geq 1$ and for all $x\in\R$,
\begin{align}
\label{Equ:Tail-of-Mnr}
\P\big(M_n^{(r)}-m_n > x\big) < C_1(1+\theta_o x_+)\e^{-\theta_o x}.
\end{align}
Since the blue process starts at $\varrho$, a similar calculation yields that, for any $n\geq 1$ and for all $x\in\R$,
\begin{align}
\label{Equ:Tail-of-Mnb}
\P\big(M_n^{(b)}-m_n > x\big) < C_1(1+\theta_o (x-\varrho)_+)\e^{-\theta_o (x-\varrho)}.
\end{align}
Therefore, by a simple union bound, for $y,z>0$ we have
\begin{align}
   \P\big(\exists n \leq \e^{\theta_o z} \text{ such that } M_n^{(r)}-m_n > y+z\big) 
   &< \e^{\theta_oz}C_1(1+\theta_oy+ \theta_oz)\e^{-\theta_o (y+z)} \nonumber\\
   &= C_1(1+\theta_oy+ \theta_oz)\e^{-\theta_oy},
\end{align}
and for $y,z>0$ and $y+z>\varrho_+$,
\begin{align}
   \P\big(\exists n \leq \e^{\theta_o z} \text{ such that } M_n^{(b)}-m_n > y+z\big) 
   &< \e^{\theta_oz}C_1(1+\theta_oy +\theta_oz -\theta_o\varrho)\e^{-\theta_oy -\theta_oz +\theta_o\varrho} \nonumber\\
   &= C_1\e^{\theta_o\varrho}(1+\theta_oy +\theta_oz -\theta_o\varrho)\e^{-\theta_oy}.
\end{align}
This completes the proof.
\end{proof}
Before we prove~\eqref{dis_3b} we establish the following estimate.
Recall that $X_n$ represents the $n$-th generation point process. Here we use the assumption that there exists $s>0$ with $\kappa(-s)<\infty$.
\begin{lem} There exist $\Lambda>0$ large enough   and $c,c_1,c_2>0$ such that for all  $n$,
\begin{align*}
   \P\big( X_n^{(r)}([-n\Lambda,\infty)) <\e^{cn} \text{ or } X_n^{(b)}([-n\Lambda,\infty)) <\e^{cn} \big) 
   &<  c_1\e^{-c_2n}.
\end{align*}
\end{lem}
\begin{proof} 
Since the argument proceeds in exactly the same way for both the red and the blue BRW, 
we omit the superscripts $(r)$ and $(b)$, which indicate colors, for notational simplicity.
As a first step we establish the following property for the associated Galton--Watson process. For all $\epsilon>0$, there exists $c_\epsilon>0$ such that for all sufficiently large $n$, we have that
\begin{align}\label{eq:GWLDP}
\P\big(X_n(\R)< \e^{(\kappa(0)-\epsilon)n}\big)\le \e^{-c_\epsilon n}.
\end{align}
In order to prove this, consider the Galton--Watson process based on the offspring law
$\Xi_K=\Xi(\R)\wedge K$ and note that for its expected offspring number we have the increasing limit $\kappa_K=\log\E[\Xi_K]\uparrow \kappa(0)=\log\E[\Xi(\R)]$. In particular, $\Xi_K$ has all moments and for $K$ sufficiently large,    
\begin{align*}
\P\big(X_n(\R)< \e^{(\kappa(0)-\epsilon)n}\big)\le\P\big(X^{(K)}_n< \e^{(\kappa_K-\epsilon/2)n}\big),
\end{align*}
where $X^{(K)}_n$ is the Galton-Watson process associated to $\Xi_K$.
Now, 
\begin{align*}
\P\big(X^{(K)}_n< \e^{(\kappa_K-\epsilon/2)n}\big)\le \sum_{\ell=1}^{\e^{(\kappa_K-\epsilon/2)n}}|\P(X^{(K)}_n=\ell)-\e^{-\kappa_Kn}\mathsf w(\e^{-\kappa_Kn}\ell)|+ \e^{-\kappa_Kn}\sum_{\ell=1}^{\e^{(\kappa_K-\epsilon/2)n}} \mathsf w(\e^{-\kappa_Kn}\ell),
\end{align*}
where $\mathsf w$ is the density of the non-trivial almost-sure limiting distribution for which we have $\e^{-\kappa_Kn}X^{(K)}_n\to W$. Performing a change of variable $k=\e^{-\kappa_Kn}\ell$, we can bound the second summand by a constant times
\begin{align*}
\P(0\le W\le \e^{-n\epsilon/2}),
\end{align*}
which tends to zero exponentially fast since $W$ has a continuous and therefor bounded density. 
For the first summand, by the local large-deviation principle~\cite[Theorem 1, Page 80]{athreya:ney:1972}, we can bound
\begin{align*}
 \sum_{\ell=1}^{\e^{(\kappa_K-\epsilon/2)n}}|\P(X^{(K)}_n=\ell)-\e^{-\kappa_Kn} \mathsf w(\e^{-\kappa_Kn}\ell)|&\le C \sum_{\ell=1}^{\e^{(\kappa_K-\epsilon/2)n}}\frac{\beta^{-n}}{\ell}+\e^{-\kappa_Kn}\beta_o^{-n}\\
 &\le C'\beta^{-n}(\kappa_K-\epsilon/2)n+\e^{-\kappa_Kn}\beta_o^{-n},
\end{align*}
where, $C,C'>0$, $\beta>1$ and we can choose $1<\beta_o<\beta$. 
Then, the second summand as well as the first summand decay exponentially. This establishes~\eqref{eq:GWLDP}.

In order to prove the main statement, represent $\Lambda=\kappa'(\theta_o)-\beta$, consider the interval
\[
I_n:= \big(n(\kappa'(\theta_o)-\beta), \infty\big)
\]
and observe that, by Markov's inequality, for $c<\kappa(0)-\epsilon$, 
\begin{align*}
\P\Big(\sum_{|v|=n}\mathbf 1\{S_v\in I_n \}<\e^{c n}, X_n(\R)\ge \e^{(\kappa(0)-\epsilon)n}\Big)
&\le \P\Big(\sum_{|v|=n}\mathbf 1\{S_v\not\in I_n\}\ge \e^{(\kappa(0)-\epsilon)n}-\e^{c n}\Big)\\
&\le \big(\e^{(\kappa(0)-\epsilon)n}-\e^{c n}\big)^{-1}\E\Big[\sum_{|v|=n}\mathbf 1\{S_v\not\in I_n\}\Big].
\end{align*} 
Now, for any $x\in\R$ and any $s>0$,  $\mathbf 1\{x\leq (\kappa'(\theta_o)-\beta)k\}\le \exp(-s(x-(\kappa'(\theta_o)-\beta)k))$ and therefore,
\begin{align*}
   \E\Big[\sum_{|v|=n}\mathbf 1\{S_v\not\in I_n\}\Big]
   \le \E\Big[\sum_{|v|=n}\e^{-s(S_v-(\kappa'(\theta_o)-\beta)n)}\Big]
   = \e^{(s(\kappa'(\theta_o)-\beta)+\kappa(-s))n}.
\end{align*}
Choosing $\beta$ sufficiently large  completes the proof.
\end{proof}

\begin{proof}[Proof of Lemma~\ref{Lem:upper-bound-of-upper-deviation} Part~\eqref{dis_3b}]
Again, since the argument proceeds in exactly the same way for both the red and the blue BRW, we omit the superscripts $(r)$ and $(b)$, which indicate colors, for notational simplicity.
For an individual $u$ we write
\[
M_n^{(u)} := \sup\{S_v-S_u\colon |v|=n+|u|, u\prec v\}.
\]
Take $\delta>0$ and $\alpha\in(0,1)$. We observe that, for all $u$ with $|u| = \floor{\alpha z}$, the sequences $\{M_n^{(u)}\}_{n \ge 1}$ are mutually independent.
Therefore,
\begin{align*}
    &\P\big(\forall n \leq \e^{\theta_o z}\colon M_n-m_n < (1-\delta)z \big)\\ 
   &\leq   \P\big(X_{\floor{\alpha z}}([-\floor{\alpha z}\Lambda,\infty)) \geq \e^{\floor{\alpha z}}\big. \\
   & \qquad\,\,\,\,\big. \text{ and } \forall u \text{ with $|u|=\floor{\alpha z}$ and $S_u>-\floor{\alpha z}\Lambda $, }  \forall n \leq \e^{\theta_o z}-\floor{\alpha z}\colon M_n^{(u)}-m_{n+\floor{\alpha z}} -\floor{\alpha z}\Lambda < (1-\delta)z  \big)\\
   &\qquad +\P\big( \calX_{\floor{\alpha z}}([-\floor{\alpha z}\Lambda,\infty)) <\e^{\floor{\alpha z}} \big)\\
   &\leq \P\big(  \forall n \leq \e^{\theta_o z}-\floor{\alpha z}\colon M_n-m_{n+\floor{\alpha z}} -\floor{\alpha z}\Lambda < (1-\delta)z  \big)^{\e^{\floor{\alpha z}}} + c_1\e^{-c_2\floor{\alpha z}}\\
   &\leq \P\big(  \forall n \leq \e^{\theta_o z}-\floor{\alpha z}\colon M_n-m_{n} < \alpha z(\Lambda +\kappa'(\theta_o))+  (1-\delta)z  \big)^{\e^{\floor{\alpha z}}} + c_1\e^{-c_2\floor{\alpha z}}.
\end{align*}
Taking $\alpha:= (\Lambda +\kappa'(\theta_o))^{-1}\delta/2$, we get that
\begin{align}
\label{Equ:Upper-bound-missing-Second-part}
    \P\big(\forall n \leq \e^{\theta_o z}\colon M_n - m_n < (1-\delta)z \big)
    \leq \P\big( \forall n \leq \e^{\theta_o z}-\floor{\alpha z}\colon 
        M_n - m_n < (1-\delta/2) z \big)^{\e^{\floor{\alpha z}}}
        + c_1 \e^{-c_2 \floor{\alpha z}} .
\end{align}
Now, by~\cite[Equation~(4.6)]{HuSh09}\footnote{ In \cite{HuSh09}, the authors assume that $\E[\Xi(\mathbb{R})^{1+\varepsilon}]<\infty$
for some $\varepsilon>0$. However, this assumption is not used in the derivation of
\cite[Equation~(4.6)]{HuSh09}, which arises only as an intermediate step in their proof.
},
we know that for any $b\in\R$
and $\epsilon>0$, all sufficiently large $\ell_1$ and all $\ell_2\in[\ell_1,2\ell_1]\cap\Z$,
\begin{align*}
   \P\big(\max_{\ell_1\leq n \leq \ell_2} \theta_o M_n- n\kappa(\theta_o)\geq - b\log \ell_1 \big) \geq \frac{\ell_2-\ell_1+1}{\ell_1^\epsilon(\ell_2-\ell_1+1)+\ell_1^{3/2-b+\epsilon}}.
\end{align*}
    Taking, $\ell_2:=2\ell_1-1$ and using $\kappa(\theta_o)=\theta_o\kappa'(\theta_o)$, this yields that for any $b\in\R$, $\epsilon>0$ and all sufficiently large $\ell_1$
\begin{align*}
   \P\left(\max_{\ell_1\leq n < 2\ell_1}  M_n- m_n\geq \left(\frac{3}{2}-b\right) \frac{1}{\theta_o}\log \ell_1 \right) \geq \frac{\ell_1^{-\epsilon}}{1+\ell_1^{1/2-b}}.
\end{align*}
Thus, taking $b:=1/2+\delta/4$ and $\ell_1:=\e^{\theta_oz}/4$, we obtain that for any $\epsilon>0$ and sufficiently large $z$,
\begin{align}
\label{Equ:lower-bound-missing-Second-part}
\P\Big(\max_{\e^{\theta_oz}/4\leq n < \e^{\theta_oz}/2}  M_n- m_n\geq (1-\delta/2) z \Big) \geq \e^{-\epsilon\theta_oz}.
\end{align}
Choosing $\epsilon<\alpha/(3\theta_o)$ and applying~\eqref{Equ:lower-bound-missing-Second-part} to~\eqref{Equ:Upper-bound-missing-Second-part} yields that for all sufficiently large $z$,
\begin{align*}
    \P\big(\forall n \leq \e^{\theta_o z}\colon M_n - m_n < (1-\delta)z \big)
    &\leq \big( 1- \e^{-\epsilon\theta_oz} \big)^{\e^{\floor{\alpha z}}}
        + c_1 \e^{-c_2 \floor{\alpha z}} \leq c_3 \e^{-\e^{\alpha z/2}} + c_1 \e^{-c_2 \floor{\alpha z}} 
       \leq c_4 \e^{- c_5 \delta z},
\end{align*}
which completes the proof.
\end{proof}

\begin{proof}[Proof of Proposition~\ref{Prop:Maximum-advantage-of-blue-before-Trz}]
Let us define the following three events
\begin{align*}
    \calA_z &:= \big\{ \e^{(1-\delta)\theta_o z} \le T^{(r)}(z) \le \e^{(1+\delta)\theta_o z}\big\},\\
    \calB_z &:=\big\{ M_n^{(b)} \le m_n + (1+2\delta) z \text{ for all } n \le T^{(r)}(z) - \floor{z}^2\big \}, \text{ and}\\
    \calC_z &:=\big\{  M_n^{(b)} \le m_n + \delta z \text{ for all } n \in ( T^{(r)}(z) - \floor{z}^2 ,T^{(r)}(z)] \big\}.
\end{align*}
We now observe what happens on the event $\calA_z\cap \calB_z \cap \calC_z $ for sufficiently large $z$. First, since the sequence $\{m_n\}_{n\ge 1}$ is eventually increasing, for sufficiently large $z$, we have
\[
m_{T^{(r)}(z)} = \sup\big\{m_n \colon n\le T^{(r)}(z)\big \}
\qquad\text{ and }\qquad
m_{T^{(r)}(z)- \floor{z}^2} = \sup\big\{m_n \colon n\le T^{(r)}(z)- \floor{z}^2 \big\}.
\]
Therefore, on the event $\calA_z\cap \calB_z \cap \calC_z $,  for all $n\le T^{(r)}(z)- \floor{z}^2$, we have
\begin{align*}
   M_n^{(b)} &\le  m_{T^{(r)}(z)- \floor{z}^2} +(1+2\delta)z\\
   & = (T^{(r)}(z)- \floor{z}^2)\kappa'(\theta_o) -\frac{3}{2\theta_o}\log (T^{(r)}(z)- \floor{z}^2) +(1+2\delta)z\\
   & = m_{T^{(r)}(z)} - \floor{z}^2\kappa'(\theta_o) -\frac{3}{2\theta_o}\log (T^{(r)}(z)- \floor{z}^2) +\frac{3}{2\theta_o}\log (T^{(r)}(z)) +(1+2\delta)z\\
   & < m_{T^{(r)}(z)} - \floor{z}^2\kappa'(\theta_o)  +\frac{3}{2} (1+\delta)z +(1+2\delta)z\\
   & < m_{T^{(r)}(z)},
\end{align*}
for sufficiently large $z$, since $\kappa'(\theta_o)>0$. Also, note that on the event $\calA_z\cap \calB_z \cap \calC_z $, for all $n$ such that $ T^{(r)}(z) - \floor{z}^2 <n \le T^{(r)}(z)$,
\[
M_n^{(b)} \le m_{T^{(r)}(z)} + \delta z .
\]
Therefore, on the event $\calA_z\cap \calB_z \cap \calC_z $, for all sufficiently large $z$, we have
\[
\sup\big\{ M_n^{(b)} \colon n\leq T^{(r)}(z)\big\} \le m_{T^{(r)}(z)}+\delta z.
\]
This implies that for all sufficiently large $z$,
\begin{align}
    \P\big(\sup\big\{ M_n^{(b)} \colon n\leq T^{(r)}(z)\big\} > m_{T^{(r)}(z)}+\delta z\big)< \P(\calA_z^\complement\cup \calB_z^\complement \cup \calC_z^\complement).
    \label{Equ:bound-by-prob-of-bad-events}
\end{align}
But, by Lemma~\ref{Lem:upper-bound-of-upper-deviation}, for all sufficiently large $z$ we have  
\[
\P(\calA_z^\complement) < C\e^{-\theta_o\delta z/2}.
\] 
Next, observe that, again by Lemma~\ref{Lem:upper-bound-of-upper-deviation} for all sufficiently large $z$,
\begin{align*}
   \P({\calA_z}\cap \calB_z^\complement) 
   \le \P\big(\exists n \leq \e^{(1+\delta)\theta_o z}\colon M_n^{(b)}-m_n > (1+2\delta)z  \big) < C\e^{-\theta_o\delta z/2}.
\end{align*}
Finally, since the red and the blue processes are independent, the estimate~\eqref{Equ:Tail-of-Mnb} together with a simple union bound, implies that for all sufficiently large $z$,
\begin{align*}
 \P(\calC_z^\complement)  
& = \P\big(\exists n \in ( T^{(r)}(z) - \floor{z}^2 ,T^{(r)}(z)]\colon M_{n}^{(b)} >  m_n + \delta z\big)\le  \floor{z}^2 C_1 \e^{-\theta_o \delta z /2 }.
\end{align*}
These estimates together with~\eqref{Equ:bound-by-prob-of-bad-events} finish the proof.
\end{proof}

\subsection{Proofs for non-coexistence}
By the bounded support of the step variables, their log-Laplace transforms
\[
\phi_r(\theta):= \log\E\big[\e^{\theta \xi^{(r)}_1}\big] \qquad \text{ and } \qquad \phi_b(\theta):= \log\E\big[\e^{\theta \xi^{(b)}_1}\big],
\]
are finite for all $\theta\in\R$. Note that, by independence, $\kappa_r(\theta):=\log\E\Big[\sum_{j=1}^{N_r}\e^{\theta \xi^{(r)}_j}\Big]=\log\E[N_r]+\phi_r(\theta)$ and $\kappa_b(\theta):=\log\E\Big[\sum_{j=1}^{N_b}\e^{\theta \xi^{(b)}_j}\Big]=\log\E[N_b]+\phi_b(\theta)$. 
\begin{proof}[Proof of Proposition~\ref{Prop:counter}]
   Let $x>0$ be such that $\P(\xi^{(r)}_1>x)\wedge\P(\xi^{(b)}_1>x)>0$. Then, using convexity of $\phi_r$ and $\phi_b$, it is immediate that
    \[
    \lim_{\theta\uparrow\infty} \kappa_r'(\theta)=\lim_{\theta\uparrow\infty} \phi_r'(\theta)= \lim_{\theta\uparrow\infty} \frac{\phi_r(\theta)}{\theta}>x \qquad \text{ and } \qquad
    \lim_{\theta\uparrow\infty} \kappa_b'(\theta)=\lim_{\theta\uparrow\infty} \phi_b'(\theta)= \lim_{\theta\uparrow\infty} \frac{\phi_b(\theta)}{\theta}>x.
    \]
Since $\phi_r'(0) = 0 = \phi_b'(0)$, by the intermediate value theorem, we choose $\theta_r$ and $\theta_b$ such that $\phi_r'(\theta_r) = x = \phi_b'(\theta_b)$. Further, since $\phi_r$ and $\phi_b$ are strictly convex, $\theta\mapsto\theta\phi_r'(\theta)-\phi_r(\theta)$ and $\theta\mapsto\theta\phi_b'(\theta)-\phi_b(\theta)$ are strictly increasing functions. This implies $\theta_r\phi_r'(\theta_r)-\phi_r(\theta_r)>0$ and also $\theta_b\phi_b'(\theta_b)-\phi_b(\theta_b)>0$. We now choose $N_r$ and $N_b$ almost~surely bounded and such that
\[
\log E[N_r] = \theta_r\phi_r'(\theta_r)-\phi_r(\theta_r)
\qquad \text{ and } \qquad
\log E[N_b] = \theta_b\phi_b'(\theta_b)-\phi_b(\theta_b).
\]
Then, $\kappa_r(\theta_r)=\theta_r\kappa'_r(\theta_r)$ and $\kappa_b(\theta_b)=\theta_b\kappa'_b(\theta_b)$ and, using~\cite[Theorem~1]{Bigg76}, we thus get that almost~surely
\[
\lim_{n\uparrow\infty} \frac{M_n^{(r)}}{n} = \kappa_r'(\theta_r)=x
\qquad \text{ and } \qquad
\lim_{n\uparrow\infty} \frac{M_n^{(b)}}{n} = \kappa_b'(\theta_b)=x,
\]   
implying~\eqref{Eq:FirstOrder}.
\end{proof}

\begin{proof}[Proof of Proposition~\ref{Prop:counter2}]
Adjusting the expected values of the offspring distributions as in the proof of Proposition~\ref{Prop:counter}, from~\cite[Theorem~1.2]{HuSh09}, we then know that almost~surely
\begin{align*}
   &\limsup_{n\uparrow \infty} \frac{M_n^{(r)}-xn}{\log n} = -\frac{1}{2\theta_r}, &  \limsup_{n\uparrow \infty} \frac{M_n^{(b)}-xn}{\log n} = -\frac{1}{2\theta_b},\\
   &\liminf_{n\uparrow \infty} \frac{M_n^{(r)}-xn}{\log n} = -\frac{3}{2\theta_r}, & \liminf_{n\uparrow \infty} \frac{M_n^{(b)}-xn}{\log n} = -\frac{3}{2\theta_b}.
\end{align*}
Note that we still have freedom to further adjust $x$ and the distribution of $\xi^{(r)}_1$ such that $-1/(2\theta_r)< -3/(2\theta_b)$, or equivalently, $\theta_b>3 \theta_r$.
To do this, we define
\[
g(\theta):= \phi_r'(\theta)-\phi_b'(3\theta)
\]
and note that, by symmetry and assuming that $\E[(\xi_1^{(r)})^2]>3$, we have that
\begin{align*}
    g(0)&=\phi_r'(0)-\phi_b'(0)=\E[\xi_1^{(r)}]- \E[\xi_1^{(b)}]=0 \quad \text{and}\quad
    g'(0)= \phi_r''(0)-3\phi_b''(0) = \E[(\xi_1^{(r)})^2]- 3\E[(\xi_1^{(b)})^2] >0.
\end{align*}
Therefore,  there exists $\alpha>0$ such that 
\[
g(\alpha)= \phi_r'(\alpha)-\phi_b'(3\alpha) >0.
\]
Setting $\theta_b=3\alpha$ and $x= \phi_b'(\theta_b)$, this implies that
\[
\phi_r'(0)=0< x  < \phi_r'(\alpha),
\]
and hence, there exists $\theta_r\in(0,\alpha)$ such that $x=\phi_r'(\theta_r)$. But this implies that $\theta_b = 3\alpha> 3 \theta_r$, as desired. As in the proof of Proposition~\ref{Prop:counter}, we now choose $N_r$ and $N_b$ almost~surely bounded and such that
\[
\log \E[N_r] = \theta_r\phi_r'(\theta_r)-\phi_r(\theta_r)
\qquad \text{ and } \qquad
\log \E[N_b] = \theta_b\phi_b'(\theta_b)-\phi_b(\theta_b).
\]
Then, by~\cite[Theorem~1.2]{HuSh09}, we obtain that almost~surely,
\begin{align*}
    \limsup_{n\uparrow \infty} \frac{M^{(r)}_n-xn}{\log n} = -\frac{1}{2\theta_r} < 
-\frac{3}{2\theta_b}=
\liminf_{n\uparrow \infty} \frac{M^{(b)}_n-xn}{\log n}, 
\end{align*}
which implies that almost~surely,
\begin{align}\label{eqn:speed_difference}
 \liminf_{n\uparrow \infty} \frac{M^{(b)}_n-M^{(r)}_n}{\log n}=  \frac{1}{2\theta_r} - \frac{3}{2\theta_b}=2c>0.
\end{align}
By symmetry, the same applies also to the left-most particles $L^{(r)}_n,L^{(b)}_n$. 
In order to finish the proof, denote by $(\Omega,\mathcal F, \P)$ the  probability space for our random experiment. Then by~\eqref{eqn:speed_difference}, there exists a nullset $\calN\in \mathcal F$, i.e., $\P(\calN)=0$, such that for all $\omega\notin\calN$ there exists $n_{\omega}$ such that for all $n\geq n_{\omega}$, $M^{(b)}_n-M^{(r)}_n>c\log n$ and $L^{(r)}_n-L^{(b)}_n>c\log n$.
Now, since almost~surely $|\xi^{(b)}_1|=1$, 
for all $n\geq n'_{\omega}= n_{\omega}\vee\exp(M/c)$, all the next $M$ sites to the right of $M^{(r)}_n$ and all the next $M$ sites to the left of $L^{(r)}_n$ must be colored blue, which means no new site will be colored red at time $(n+1)$. This implies that all the sites not in $[L^{(r)}_{n'_{\omega}}, M^{(r)}_{n'_{\omega}}]$ must be colored blue eventually. This completes the proof.
\end{proof}

\begin{proof}[Proof of Theorem~\ref{Thm:counter}] 
     The first part of the proof proceeds precisely as the proof of Proposition~\ref{Prop:counter2} above. However, we need additional arguments that replace the last steps in that proof since we heavily used the fact that the blue BRW leaves no holes. In order to overcome this, we restrict our attention to the case where 
$\xi^{(b)}_1$ is uniformly distributed on $\{-2,-1,1,2\}$, and where 
$N_b\ge 2$ and bounded almost~surely with $\E[N_b] < 4$.
 This is only to simplify calculation as much as possible and we believe that generalizations to other random-walk and offspring distributions are possible. The requirement $\E[N_b] < 4$ ensures the existence of $\theta_b$, and is therefore crucial for our argument. To see this, observe that by \cite[Proposition~3.3.2]{Ghosh22}, $\theta_b$ exists iff 
     \[
     \lim_{\theta\uparrow\infty}\kappa_b(\theta) - \theta \big(\lim_{x\uparrow\infty} \kappa'(x)\big)<0.
     \]
     Note that
     \[
     \lim_{x\uparrow\infty}\kappa_b'(x) = \lim_{x\uparrow\infty}\phi_b'(x)
     =\lim_{x\uparrow\infty}\frac{\phi_b(x)}{x}=2.
     \]
     Now, 
     \begin{align*}
         \lim_{\theta\uparrow\infty}\kappa_b(\theta) - 2\theta
         &= \lim_{\theta\uparrow\infty}\log\E[N_b] -\log 4 + \log (1+ \e^{-\theta} +\e^{-3\theta} +\e^{-4\theta})= \log\E[N_b] -\log 4 <0.
     \end{align*}
     This implies that there exists $\theta_b>0 $ such that $\kappa_b(\theta_b)=\theta_b\kappa_b'(\theta_b)$.

     We now choose $\xi^{(r)}_1$ to be uniformly distributed on $\{-M,-1,1,M\}$, for some $M\geq 2$. Then for any $\theta>0$,
     \[
     \phi_r(\theta)=  \log\Big( \frac{1}{4}\e^{\theta} + \frac{1}{4}\e^{-\theta}+ \frac{1}{4}\e^{M\theta}+ \frac{1}{4}\e^{-M\theta}\Big),
     \]
     and, therefore,
     \[
     \phi_r'(\theta)= \frac{\e^{\theta}-\e^{-\theta}+M\e^{M\theta}-M\e^{-M\theta}}{\e^{\theta}+\e^{-\theta}+\e^{M\theta}+\e^{-M\theta}},
     \]
     which goes to $\infty$ as $M\uparrow\infty$, for every $\theta>0$. We can then choose $M$ large enough such that $\phi_r'(\theta_b/3)> \kappa_b'(\theta_b)$. Since $\phi_r'(0)=0$, by the intermediate value theorem, there exists $\theta_r\in(0,\theta_b/3)$ such that  $\phi_r'(\theta_r)= \kappa_b'(\theta_b)$. Then, as in the proof of  Proposition~\ref{Prop:counter}, observing that $\theta_r\phi_r'(\theta_r) -\phi_r(\theta_r)>0$,  we can choose $N_r$, almost~surely bounded such that
     \[
     E[N_r]= \theta_r\phi_r'(\theta_r) -\phi_r(\theta_r). 
     \]
     This makes $\kappa_r(\theta_r)=\theta_r\kappa_r'(\theta_r)$ and, by construction, we have $3\theta_r<\theta_b$. Therefore, both~\eqref{Eq:FirstOrder} and~\eqref{eqn:speed_difference} holds for $M^{(r)}_n,M^{(b)}_n$, respectively by symmetry also for $L^{(r)}_n,L^{(b)}_n$.
   
    Moreover, as in the final step in the proof of Proposition~\ref{Prop:counter2}, there exists a nullset $\calN\in \mathcal F$ such that for all $\omega\notin\calN$ there exists $n_{\omega}$ such that for all $n\geq n_{\omega}$, $M^{(b)}_n-M^{(r)}_n>c\log n$ and $L^{(r)}_n-L^{(b)}_n>c\log n$. 
Now, since almost~surely $|\xi^{(b)}_1|\le 2$, 
for all $n\geq n'_{\omega}= n_{\omega}\vee\exp(M/c)$, at least every other site within the next $M$ sites to the right of $M^{(r)}_n$, respectively to the left of $L^{(r)}_n$, must be colored blue. 
In particular, with 
\begin{align*}
\tau^{(b)}(k)=\inf\{n\ge 0\colon M^{(b)}_n\in \{k, k+1\}\}\qquad\text{and}\qquad\tau^{(r)}(k)&=\inf\{n\ge 0\colon M^{(r)}_n\in \{k, k+1\}\},
\end{align*}
we have that $\{\tau^{(b)}(k)<\tau^{(r)}(k)\}$ for all but finitely many $k\ge 0$ with probability one. 

Further, in order to ensure that blue leaves no holes, the idea is the following. Due to the second-order advantage of the blue BRW, at the time when blue hits a position $k$, it has some time before the red BRW catches up. This time is enough to also color a previously uncolored neighboring site of $k$. In  order to implement this precisely, consider the events 
\begin{align*}
H_{k}=\{\text{the set }\{k,k+1\}\text{ is entirely colored blue}\}.
\end{align*}
By symmetry, it suffices to show that $H_k$ occurs for all but finitely many $k\ge 0$.  
In view of the Borel--Cantelli lemma, it suffices to show that 
\begin{align}\label{Equ:a1}
\sum_{k\ge 0}\P\big(H^{\complement}_k, M^{(b)}_{\tau^{(b)}(k)}-M^{(r)}_{\tau^{(b)}(k)}\ge c\log \tau^{(b)}(k),\tau^{(b)}(k)<\tau^{(r)}(k) \big)
<\infty.
\end{align}
In words, at the time $\tau^{(b)}(k)$, blue hits the set $\{k,k+1\}$ for the first time and does this before red. Additionally, the right-most blue particle is at least $c\log \tau^{(b)}(k)$ steps to the right of the right-most red particle, but still does not manage to hit the entire set $\{k,k+1\}$ before red. But this is highly unlikely. 

Indeed, the event in~\eqref{Equ:a1} implies that a BRW started in position $M^{(b)}_{\tau^{(b)}(k)}\in \{k,k+1\}$ does not reach $\{k,k+1\}\setminus M^{(b)}_{\tau^{(b)}(k)}$ in $c\log \tau^{(b)}(k)$ time steps. Hence, the left-hand side of~\eqref{Equ:a1} can be bounded by 
\begin{align}\label{Equ:a2}
2\sum_{k\ge 1}\P\big(\forall i\le c\log k\colon X^{(b)}_i(1)=0\big),
\end{align}
where we use that $\tau^{(b)}(k) \geq \floor{k/2}$ almost~surely.
In order to bound~\eqref{Equ:a2} note that, at time $\lfloor\ell^{1/3}\rfloor$ with  $\ell= \floor{c\log k}$, blue has at least $2^{\lfloor\ell^{1/3}\rfloor}$  many particles in positions contained in $K_{\ell}=\{ -2\lfloor\ell^{1/3}\rfloor,\ldots, 2\lfloor\ell^{1/3}\rfloor \}$. 
We show that, even when starting independent random walks from these positions, 
each with increment distribution equal to that of $\xi^{(b)}_1$, we reach $1$ 
with high probability. For branching random walks, this probability is even higher.
For a random walk, the probability of touching or crossing $0$ decreases as the 
starting position moves farther away from $0$. Furthermore, when it touches or crosses $0$ without having visited $1$ yet,
it can reach $1$ in the next step with probability $1/4$.

Now, let $\{\mathsf S_n\}_{n\geq 0}$ be a random walk with with increment distribution equal to that of $\xi^{(b)}_1$ and $\mathsf S_0= 2\lfloor\ell^{1/3}\rfloor$. By~\cite[Lemma~4.1]{AiSh10},
\begin{align*}
    \P\big(\forall n\leq \ell-\lfloor\ell^{1/3}\rfloor\colon \mathsf S_n\geq 1\big) \leq \frac{2\lfloor\ell^{1/3}\rfloor}{(\ell-\lfloor\ell^{1/3}\rfloor)^{1/2}}<\epsilon,
\end{align*}
for all sufficiently large $\ell$, for some fixed $\epsilon \in (0,1)$.
 Therefore, we obtain that
\begin{align*}
    \P\big(\forall i\le c\log k\colon X^{(b)}_i(1)=0\big) \leq \big(1- (1-\epsilon)/4 \big)^{2^{\lfloor\ell^{1/3}\rfloor}}
    = \big(1- (1-\epsilon)/4 \big)^{2^{\lfloor{\floor{c\log k}}^{1/3}\rfloor}}. 
\end{align*}
Finally, the fact that the right-hand side is summable in $k$ can be seen by a suitable change of variable. 
\end{proof}

\section*{Acknowledgement} 
We thank Johannes B\"aumler for his input on the occasion of the workshop {\em Recent Advances in Evolving and Spatial Random Graphs} in Bayrischzell in June 2025 as well as Bastien Mallein for indicating a crucial reference. We also thank Maria Deijfen for pointing us to several relevant references.
BJ received support by the Leibniz Association within the Leibniz Junior Research Group on \textit{Probabilistic Methods for Dynamic Communication Networks} as part of the Leibniz Competition (grant no.\ J105/2020). The research of BJ is funded by Deutsche Forschungsgemeinschaft (DFG) through the SPP2265 within the Project P27.


\begin{thebibliography}{DGH23}

\bibitem[AN72]{athreya:ney:1972}
K.~Athreya and P.~Ney.
\newblock {\em Branching Processes}, volume 196 of {\em Grundlehren der
  mathematischen Wissenschaften}.
\newblock Springer-Verlag, Berlin, Heidelberg, New York, 1972.

\bibitem[AS10]{AiSh10}
E.~A\"id\'ekon and Z.~Shi.
\newblock Weak convergence for the minimal position in a branching random walk:
  a simple proof.
\newblock {\em Periodica Mathematica Hungarica}, 61(1-2):43--54, 2010.

\bibitem[Aï13]{Aidekon2013}
É. Aïdékon.
\newblock Convergence in law of the minimum of a branching random walk.
\newblock {\em The Annals of Probability}, 41(3A):1362--1426, 2013.

\bibitem[Big76]{Bigg76}
J.~D. Biggins.
\newblock The first- and last-birth problems for a multitype age-dependent
  branching process.
\newblock {\em Advances in Applied Probability}, 8(3):446--459, 1976.

\bibitem[Big90]{Biggins1990}
J.~D. Biggins.
\newblock The central limit theorem for the supercritical branching random
  walk, and related results.
\newblock {\em Stochastic Processes and their Applications}, 34:255--274, 1990.

\bibitem[Big92]{Biggins1992}
J.~D. Biggins.
\newblock Uniform convergence of martingales in the branching random walk.
\newblock {\em The Annals of Probability}, 20(1):137--151, 1992.

\bibitem[Che13]{Chen13}
X.~Chen.
\newblock Waiting times for particles in a branching brownian motion to reach
  the rightmost position.
\newblock {\em Stochastic Processes and their Applications}, 123(8):3153--3182,
  2013.

\bibitem[DGH23]{GantertDyszewskiHofelsauer2023}
P.~Dyszewski, N.~Gantert, and T.~H{\"o}felsauer.
\newblock The maximum of a branching random walk with stretched exponential
  tails.
\newblock {\em Annales de l'Institut Henri Poincar{\'e} (B) Probabilit{\'e}s et
  Statistiques}, 59(2):539--562, 2023.

\bibitem[DV23]{DeVi23}
M.~Deijfen and T.~Vilkas.
\newblock Competition on $\mathbb{Z}^d$ driven by branching random walk.
\newblock {\em Electronic Communications in Probability}, 28(15):1--11, 2023.

\bibitem[GH18]{GantertHofelsauer2018}
N.~Gantert and T.~H{\"o}felsauer.
\newblock Large deviations for the maximum of a branching random walk.
\newblock {\em Electronic Communications in Probability}, 23(34):1--12, 2018.

\bibitem[Gho22]{Ghosh22}
P.~P. Ghosh.
\newblock {\em A Last Progeny Modified Branching Random Walk}.
\newblock Phd thesis, Indian Statistical Institute, 2022.

\bibitem[Ham74]{Hammersley1974}
J.~M. Hammersley.
\newblock Postulates for subadditive processes.
\newblock {\em The Annals of Probability}, 2(4):652--680, 1974.

\bibitem[HS09]{HuSh09}
Y.~Hu and Z.~Shi.
\newblock Minimal position and critical martingale convergence in branching
  random walks, and directed polymers on disordered trees.
\newblock {\em The Annals of Probability}, 37(2):742--789, 2009.

\bibitem[Kin75]{Kingman1975}
J.~F.~C. Kingman.
\newblock The first birth problem for an age-dependent branching process.
\newblock {\em The Annals of Probability}, 3(5):790--801, 1975.

\bibitem[LP15]{LouidorPerkins2015}
O.~Louidor and W.~Perkins.
\newblock Large deviations for the empirical distribution in the branching
  random walk.
\newblock {\em Electronic Journal of Probability}, 20(18):1--19, 2015.

\bibitem[LS87]{LaSe87}
S.~P. Lalley and T.~Sellke.
\newblock A conditional limit theorem for the frontier of a branching brownian
  motion.
\newblock {\em The Annals of Probability}, 15(3):1052--1061, 1987.

\bibitem[Mad17]{Mada17}
T.~Madaule.
\newblock Convergence in law for the branching random walk seen from its tip.
\newblock {\em Journal of Theoretical Probability}, 30(1):27--63, 2017.

\bibitem[Mal16]{Mallein2016}
B.~Mallein.
\newblock Asymptotic of the maximal displacement in a branching random walk.
\newblock {\em Graduate Journal of Mathematics}, 125(10):92--104, 2016.

\bibitem[Shi19]{ShiWanlin2019}
W.~Shi.
\newblock A note on large deviation probabilities for empirical distribution of
  branching random walks.
\newblock {\em Statistics \& Probability Letters}, 147:18--28, 2019.

\bibitem[Zha22]{Zhang2022}
S.~Zhang.
\newblock On large-deviation probabilities for the empirical distribution of
  branching random walks with heavy tails.
\newblock {\em Journal of Applied Probability}, 59(2):471--494, 2022.

\end{thebibliography}
\end{document}